\documentclass[final,1p,times]{elsarticle}
\usepackage{graphics}
\usepackage{graphicx}
\usepackage{epsfig}
\usepackage{epstopdf}
\usepackage{pdfsync}
\usepackage{amssymb}
\usepackage{amsthm}
\usepackage{amsmath}
\usepackage{tikz}
\usepackage{subfigure} 
\usetikzlibrary{shapes,arrows,positioning}
\usetikzlibrary{decorations.markings} 
\usetikzlibrary{plotmarks}
\usepackage{epstopdf}
\usepackage{epsfig}
\newtheorem{thm}{Theorem}

\usepackage{setspace}
\usepackage{color}
\usepackage{graphicx}
\usepackage{caption}
\usepackage{subfigure}

\begin{document}

\doublespacing

\begin{frontmatter}

%% Title, authors and addresses

%% use the tnoteref command within \title for footnotes;
%% use the tnotetext command for the associated footnote;
%% use the fnref command within \author or \address for footnotes;
%% use the fntext command for the associated footnote;
%% use the corref command within \author for corresponding author footnotes;
%% use the cortext command for the associated footnote;
%% use the ead command for the email address,
%% and the form \ead[url] for the home page:
%%
%% \title{Title\tnoteref{label1}}
%% \tnotetext[label1]{}
%% \author{Name\corref{cor1}\fnref{label2}}
%% \ead{email address}
%% \ead[url]{home page}
%% \fntext[label2]{}
%% \cortext[cor1]{}
%% \address{Address\fnref{label3}}
%% \fntext[label3]{}

\title{An Optimal Query Assignment Policy for Wireless  Sensor Networks}

%% use optional labels to link authors explicitly to addresses:
%% \author[label1,label2]{<author name>}
%% \address[label1]{<address>}
%% \address[label2]{<address>}

\author[label1]{Mihaela Mitici}
 \ead{M.A.Mitici@utwente.nl}
\author[label2]{Martijn Onderwater}
\ead{ M.Onderwater@cwi.nl}
\author[label1,label3]{Maurits de Graaf}
 \ead{M.deGraaf@utwente.nl}
\author[label1]{Jan-Kees van Ommeren}
 \ead{J.C.W.vanOmmeren@utwente.nl}
\author[label1]{Nico van Dijk}
 \ead{N.M.vanDijk@utwente.nl}
\author[label1]{Jasper Goseling}
 \ead{J.Goseling@utwente.nl}
\author[label1]{Richard J. Boucherie}
 \ead{R.J.Boucherie@utwente.nl}

%\address{}
\address[label1]{Department of Applied Mathematics, University of Twente,  P.O.Box 217, 7500 AE Enschede, The Netherlands}
\address[label2]{Centrum voor Wiskunde en Informatica, P.O.Box 94079, NL-1090 GB Amsterdam, The Netherlands }
\address[label3]{Thales B.V., P.O.Box 88 1270, Huizen, The Netherlands}

\begin{abstract}
A trade-off between two QoS requirements of wireless sensor networks: query waiting time and validity (age) of the data feeding the queries, is investigated. 
We propose a Continuous Time Markov Decision Process with a drift that trades-off between the two QoS requirements by assigning incoming queries to the wireless sensor network or to the database.
%The continuous character of the process makes the problem computationally intractable. 
To compute an optimal assignment policy, we argue, by means of non-standard uniformization, a discrete time Markov decision process, stochastically equivalent to the initial continuous process.
We determine an optimal query assignment policy for the discrete time process by means of dynamic programming.
Next, we assess numerically the performance of the optimal policy and show that it outperforms in terms of average assignment costs three other heuristics, commonly used in practice.
Lastly, the optimality of the our model is confirmed also in the case of real query traffic, where our proposed policy  achieves significant cost savings compared to the heuristics.
\end{abstract}

\begin{keyword}
 Wireless Sensor Networks\sep Markov Decision Processes, Quality of Service

%% MSC codes here, in the form: \MSC code \sep code
%% or \MSC[2008] code \sep code (2000 is the default)

\end{keyword}

\end{frontmatter}

%%
%% Start line numbering here if you want
%%
% \linenumbers

%% main text
\section{Introduction}
\label{Introduction}

Wireless sensor networks (WSNs) are commonly used to sense environmental phenomena such as forest fire detection, intruder detection and indoor environmental control \cite{akyildiz2002wireless}. 
The sensed data is stored in databases, from which queries can be processed at a later stage.

The increased computing capabilities of modern sensor networks have enabled the WSNs to become an integrated platform on which local query processing is performed. Consequently, not only the Database (DB) is able to store and process queries, but also the sensors within the WSN.
Letting the WSN to solve queries, however, poses Quality of Service (QoS) challenges. For example, sensors can answer queries with the most recently sensed data. But always directing the queries to the WSN can overload the network and lead to high query waiting times. A trade-off arises between solving the queries within the WSN with the most recently acquired data and the time queries wait to be processed.

In recent years, studies on sensor networks have focused mainly on energy efficient data transmission~\cite{chen2004qos,lei2009generic,munir2009mdp} and the traffic was assumed to have unconstrained delivery requirements. However, growing interest in applications with specific QoS requirements has created additional challenges. We refer to ~\cite{chen2004qos,krishnamachari2005networking} for an extensive outline of WSN specific QoS requirements. 
The literature reveals related work on QoS-based routing protocols within the sensor network. Most such protocols satisfy end-to-end packet delay ~\cite{he2003speed} or data reliability requirements ~\cite{deb2003reinform,stann2003rmst} or a trade-off between the two ~\cite{felemban2005probabilistic}.
However, little work exists on QoS guarantees in the field of  sensor query monitoring, as addressed in this paper. In ~\cite{yao2003query} a query optimizer is used to satisfy query delay requirements. In ~\cite{khoury2010corona} the authors use data validity restrictions to specify how much time is allowed to pass since the last sensor acquisition so that the sensors are not activated, but previously sensed data is used.  

This paper addresses the trade-off between two QoS requirements commonly encountered in WSNs: waiting time for queries processed by the WSN and validity (age) of the data when queries are processed by a database (DB). 
We consider a system consisting of a DB and a WSN, both able to solve queries (see Figure 1). 
\begin{figure}\label{System}
\centering
\includegraphics[height=2.9cm]{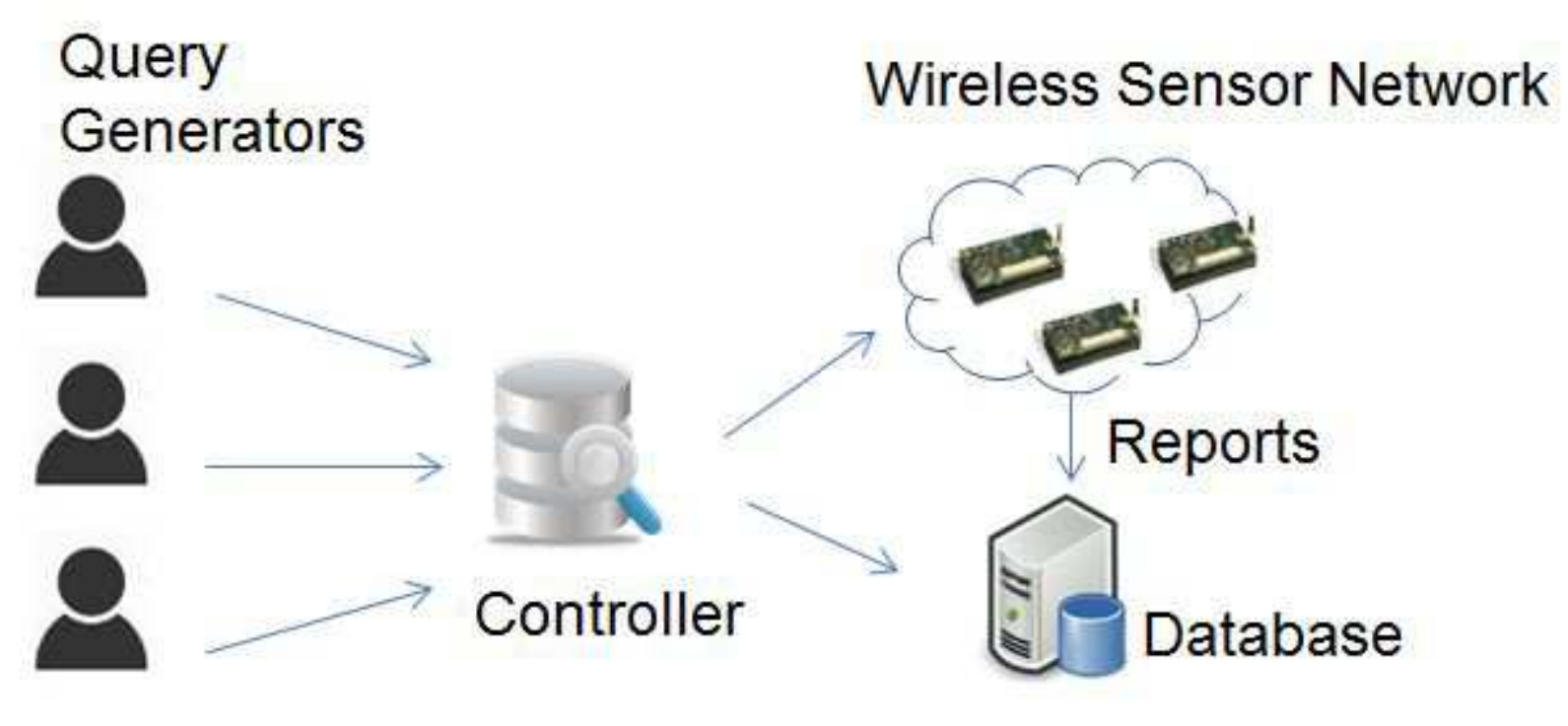}
\caption{WSN seen as an integrated platform, where queries generated by the end-users can be solved either by the WSN or by the DB. }
 \end{figure}

We assume that the DB solves queries instantaneously, since the time required to fetch data from the DB is negligible compared to the time a query is processed by the WSN.
The processor sharing type of service assumed for the WSN reflects the IEE 802.15.4 MAC design principle of distributing the processing capacity fairly among the jobs simultaneously present in the network. 
Processor sharing for WLAN was assumed in \cite{khoury2010corona} and validated  by simulation in \cite{litjens2003performance}.

Queries arrive at the system according to a Poisson process and are assigned  for processing either to the WSN or to the DB.
A WSN assignment increases the load of the network and results in large query waiting times. If queries are sent to the DB, the data provided to the queries may be outdated, as the age of the stored data increases in time. The fact that  the quality of the stored data deteriorates in time is an  essential feature of our system and will impose technical complications, as to be seen in the next section. %The continuous evolution of the age of the data will impose technical complications when determining an optimal assignment policy, as to be seen in the next section.}
%An is that the quality of the data stored in the DB deteriorates over time, i.e. the age of the stored data increases over time. In other words, the age of the stored data evolves continuously over time. 
We are interested in finding an optimal query assignment strategy such that the query waiting time and the age of the data provided to the queries are minimized. 

The query assignment problem presented above is formulated as a Continuous Time Markov Decision Process (CTMDP) with a drift. 
The continuous character of the process, and in particular,  the fact that the age component of the process evolves continuously in time, makes the problem non-standard and computationally intractable, i.e. the standard way of deriving an optimal policy recursively using dynamic programming is not applicable.
Therefore, for computational reasons, we argue a discrete state and time Markov decision process. First, we propose a non-standard exponentially uniformized Markov decision process, which we show to be  \textit{stochastically equivalent} to the original CTMDP with a drift. 
However, the exponentially uniformized process still contains the age as a continuous state component.
Therefore, for further computational tractability, we argue a  discrete time and state Markov decision process. We then determine an optimal query assignment policy for the discrete time and state process by means of stochastic dynamic programming. Finally, we argue and numerically illustrate that the optimal policy also holds for the original CTMDP with a drift.

In addition, the performance of our optimal strategy is assessed numerically. We show that it outperforms in terms of average assignment costs three other feasible assignment heuristics, commonly used in practice.
Lastly, also in the case of real query traffic, our proposed policy achieves significant cost savings compared to the heuristics.
The results provide useful insight into deriving simple assignment strategies that can be readily used in practice.

The paper is structured as follows. In Section \ref{Model Formulation}, we describe the model of the query assignment problem and define it as a Discrete Time and Space Markov Decision Problem. In Section \ref{Numerical results}, we  assess numerically the performance of our proposed assignment policy and compare it with other feasible heuristics. Results for real traffic queries under our proposed policy and three heuristics are also presented. Concluding remarks are stated in Section \ref{Conclusion and Future work}.

\section{Model Formulation}
\label{Model Formulation}

In this section we introduce formally the query assignment problem. 
In section \ref{Continuous Time Markov Decision Process}, we define the query assignment problem as a Continuous Time Markov Decision Process (CTMDP) with a drift. Next, we construct an exponentially uniformized Markov Decision Process in section \ref{Exponential  Markov Decision Process}. We show that the uniformized Markov decision process and the continuous time  process (section  \ref{Continuous Time Markov Decision Process}) are stochastically equivalent.
This leads to the formulation of the assignment problem as a Discrete Time and Space Markov Decision Problem in section \ref{Discrete Time Markov Decision Process}.

\subsection{Model Description} \label{Model Description}
The system consists of a service facility (WSN) with processor sharing capabilities and a storage facility (DB).  Figure \ref{NEWSystem} shows the proposed model. 

Two types of jobs: queries and reports, arrive at the system according to a Poisson process. 
Queries arrive at rate $\lambda_1$. Reports arrive at rate $\lambda_{2}$. Reports are requests issued to the WSN to sense the environment and send the data to the DB. Reports update, therefore, the DB.
The service requirements of the jobs are exponentially distributed with parameter $\mu$, independently of the job type.  To ensure that the system is stable, we assume that $\lambda_2 <\mu$.
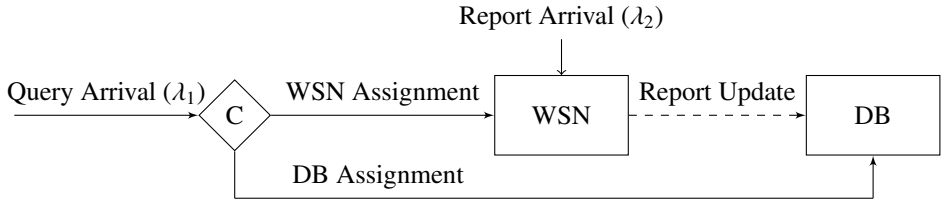
\begin{figure}[ht]
\tikzstyle{block} = [draw, fill=blue!0, rectangle, 
    minimum height=3em, minimum width=5em]
\tikzstyle{C} = [draw, fill=blue!0, diamond, node distance=2.9cm]
\tikzstyle{input} = [coordinate]
\tikzstyle{input2} = [coordinate]
\tikzstyle{output} = [coordinate]
\tikzstyle{pinstyle} = [pin edge={to-,thin,black}]

% The block diagram code is probably more verbose than necessary
\begin{tikzpicture}[auto, node distance=2cm,>=latex']
    % We start by placing the blocks
    \node [input, name=input] {};
    \node [C, right of=input] (C) {C};
    \node [block, right of=C,   pin={[pinstyle]above:Report Arrival ($\lambda_{2}$)},node distance=4.3cm] (WSN) {WSN};
    \node [block, right of=WSN, node distance=4.1cm] (DB) {DB};
    % We draw an edge between the WSN and DB block to 
    % calculate the coordinate u. We need it to place the measurement block. 
    \draw [->, dashed] (WSN) -- node[name=u] {Report Update} (DB);
    \node [output, right of=DB] (output) {};
   %\node [block, below of=u] (measurements) {Measurements};

    % Once the nodes are placed, connecting them is easy. 
    \draw [draw,->] (input) -- node {Query Arrival ($\lambda_{1}$)} (C);
    \draw [->] (C) -- node {WSN Assignment} (WSN);
    \draw [->](11.3,-1.1)-- node {DB Assignment \:\:\:\:\: \:\:\:\:\: \:\:\:\:\: \:\:\:\:\: \:\:\:\:\: \:\:\:\:\: \:\:\:\:\: \:\:\:\:\: \:\:\:\:\: \:\:\:\:\: \:\:\:\:\:} (DB);
    \draw[-](11.3,-1.1)--  (2.9,-1.1);
    \draw[-](2.9,-1.1)--(C);
\end{tikzpicture}
\caption{Proposed model incorporating a controller (C ), the database (DB) and the wireless sensor network (WSN). The DB solves queries assigned by the controller. The WSN solves reports and queries assigned by the controller.}
\label{NEWSystem}
\end{figure}

Incoming queries are handled by a controller which assigns them either to the DB or to the WSN. When assigned to the DB, queries are immediately answered with stored data. Assigned to the WSN, queries wait to receive the sensed data, sharing the service with the other jobs present in the network. 
The query assignment decision is based on the trade-off between the query waiting time and the age of the stored data, upon a DB assignment.

Several assignment heuristics are already used in practice, three of which we analyze in section \ref{Numerical results}. We are interested in investigating whether no further improvements in the  expected assignment cost can be achieved. As such, we formulate our problem as a Markov decision problem and find an optimal assignment policy that achieves a trade-off between the query waiting time and the validity (age) of the data provided to the queries.

\subsection{Stochastic Dynamic Programming Formulation}
 
As mentioned earlier, in order to make the assignment problem computationally tractable, we will follow three steps, to be found in sections 2.2.1, 2.2.2 and 2.2.3.

\subsubsection{Continuous Time Markov Decision Process with a Drift} \label{Continuous Time Markov Decision Process}
The system in section \ref{Model Description} can be described as a CTMDP with a drift as stated below.
For an introduction to CTMDP with a drift we refer to  \cite{hordijk1984discretization}.

Firstly, at any point in time, the system is completely described by the number of queries, reports and the age of the data stored in the DB. Thus,  the state space of the problem is defined as follows:
\begin{itemize}
\item State space $S= \mathbb{N}_{0} \times \mathbb{N}_{0} \times (0,\infty)$, 
where $(i,j,t) \in S$denotes the state in which there are $i$ queries in the WSN, $j$ reports and $t$ the time since the last report completion (age of the stored data).
\end{itemize}

Upon a query arrival, the controller assigns the query either to the DB or to the WSN. The action space is, thus, defined as:
\begin{itemize}
\item Action: the controller takes an action $d$ from the action space $D=\{DB, WSN\}$, where $d=DB$ denotes a DB assignment and $d=WSN$ denotes a WSN assignment. 
\end{itemize}
We define a policy $\pi$ to be a mapping from the state space $S \rightarrow D$, which specifies the action $d \in D$ the controller takes when the system is in state $(i,j,t) \in S$ and a query arrival occurs. We make the natural assumption that this policy is left-continuous in the age component $t$, which allows for threshold-type of assignment policies of the form $t>T$, where T is a threshold.

The system has a state transition upon a query arrival, a report arrival, a query completion or a report completion. The rates at which these events happen  are as follows:
\begin{itemize}
\item The transition rates, when in state $(i,j,t) \in S$ and action $d \in D$ is taken:
\begin{align}\label{TransitionRates}
 q^{d}[(i,j,t),(i,j,t)'] =
\begin{cases}
\lambda_{1},  & (i,j,t)'=(i+1,j,t), \:\:\:\:\:\:\:\:\:\:\:\:\:\:\: d = WSN  \\
\lambda_{1},  & (i,j,t)'=(i,j,t), \:\:\:\:\:\:\:\:\:\:\:\:\:\:\:\:\:\:\:\:\:\: d = DB  \\
\lambda_{2},  & (i,j,t)'=(i,j+1,t) \\
\mu \phi_{1}(i,j), &(i,j,t)'=(i-1,j,t), i>0 \\
\mu \phi_{2}(i,j), &(i,j,t)'=(i,j-1,0), j>0 \\
\end{cases}
\end{align}
\end{itemize}
with $\phi_{1}(i,j)=\frac{i}{i+j}, \phi_{2}(i,j)=\frac{j}{i+j}$ indicating the Processor Sharing service discipline assumed for the WSN.
The first  line of (\ref{TransitionRates}) models a query arrival under action $d=WSN$, i.e. the query is assigned to the WSN for processing. The state space illustrates an increment in the number of queries from $i$ to $i+1$. 
The second line of (\ref{TransitionRates}) models a query arrival under action $d=DB$, i.e. the query is assigned to the DB. In this case, the query is solved immediately, no changes occur in the number of the queries and reports in the system. 
The third line of (\ref{TransitionRates}) models a report arrival. The state of the system illustrates an increment in the number of reports.
The fourth line of (\ref{TransitionRates}) models a query completion at the Processor sharing rate $\phi_{1}(i,j)=\frac{i}{i+j}$. The number of queries in the system is decremented to $i-1$. Lastly, the fifth line of (\ref{TransitionRates}) models a report completion at the  Processor sharing rate  $\phi_{2}(i,j)=\frac{j}{i+j}$. The age of the stored data is reset to zero as the report updates the DB with the most recently sensed data.

The above Markov decision process has a deterministic drift for the age component, $t$.
This increases linearly as long as no report is completed. 
Also, we consider two types of decisions. Firstly, the decision to assign an incoming query to the DB affects only the infinitesimal generator of the Continuous Time Markov Decision Process (see second line of (\ref{TransitionRates})). Secondly, the decision to assign a query to the WSN affects both the infinitesimal generator and determines a change in the state of the system (see first line of  (\ref{TransitionRates})).

The dynamics of this controlled Markovian decision process are uniquely determined by its infinitesimal generators (see, for instance, \cite{dynkin1965markov}). In the case of our system described above, under action $d$, the generator is specified, for any function $f:S \times S \times (0,\infty) \rightarrow \mathbb{R}$, as follows:
\begin{align}\label{Agenerator}
\textbf{A}^{d}f(i,j,t)=\sum \limits_{(i,j,t)'} q^{d}[(i,j,t),(i,j,t)'] \cdot f[(i,j,t)']+\frac{d}{dt}f(i,j,t)
\end{align}
The generator stated in (\ref{Agenerator}) shows that, over time, a jump to a new state $(i,j,t)'$ occurs at rate $q^{d}$ or no jump occurs and time increases.

The cost of the system is two-folded. Firstly, we consider the cost $i$ of having $i$ queries waiting within the WSN to be processed. This cost gives an overview of the load of the WSN over time. Having a large number of queries in the WSN incurs penalties as the queries need to wait more to be processed. Secondly, we consider an instantaneous cost incurred every time a query is solved by the DB. We incur a penalty for each time unit the age of the stored data exceeds a given threshold $T$. The two costs illustrate the trade-off between the waiting time of the queries within the WSN and the age of the DB data provided to the queries.
Formally, this is expressed by:
\begin{itemize}
\item Cost: when in state $(i,j,t)$, a cost rate $i$ for the queries waiting in the WSN and an instantaneous cost $(t-T)^+$, where $x^{+}=\max(x,0)$, upon a DB assignment.
\end{itemize}
The cost function assumes no explicit communication times. When queries are assigned to the WSN, the communication time is implicitly included in the time the query waits to be processed. In the case of a DB assignment, the processing and communication time are negligible compared to the time a query is processed within the WSN. Therefore, we assume a query is immediately processed when assigned to the DB.

\subsubsection{Exponentially Uniformized Markov Decision Process} \label{Exponential Markov Decision Process}

The continuous character of the process described in subsection \ref{Continuous Time Markov Decision Process}, and in particular, the continuous age component of the process evolving in time, make the problem computationally intractable, i.e. the standard way of deriving an optimal policy recursively by using dynamic programming is not applicable for a Continuous Time Markov Decision Process with a drift.

To this end, uniformization, a method commonly used to make a continuous time MDP computational tractable, is not applicable due to the drift (age component evolving in time) of our process. Uniformization, as introduced in \cite{jensen1953markoff}, is a well-known technique used to transform a continuous time Markov jump process (see \cite{gikhman2004theory, van1990simple}) into a discrete time Markov process.  When the state is also discrete, it is referred to as a discrete time Markov chain. 

%This standard uniformization is not applicable for processes with a state component continuously evolving in time, as our case. 
In \cite{hordijk1984discretization} and \cite{van1996time}, time discretization is applied to continuous time Markov decision processes with a drift component evolving in time. This is a somewhat similar method to uniformization. Time discretization is, however, an approximative method and leads to technical weak convergence and not exact results for computational purposes, as aimed in this paper. 
Therefore, to be able to compute an optimal query assignment policy, in what follows below we construct an exponentially uniformized Markov Decision Process, and show it to be stochastically equivalent to the initial CTMDP with a drift. This implies that the two processes are the same in terms of expected assignment costs and policies.
We can then argue both a discrete time and state Markov decision process which is computational tractable, i.e. we are able to compute an optimal assignment policy. We argue and show numerically that this policy also holds for the CTMDP with a drift (section \ref{Continuous Time Markov Decision Process}).

We now uniformize the CTMDP with a drift described in  subsection \ref{Continuous Time Markov Decision Process} as follows:

Let $B$ be an arbitrarily large finite number such that $B \geq \lambda_{1}+\lambda_{2}+\mu$. At exponential times with parameter $B$, the system will have a transition. 
Denote by $s$ the exponential realization time of this transition.
Then, given the state space assumed in subsection \ref{Continuous Time Markov Decision Process} and the transition realization of duration $s$, the transition probabilities under action $d \in D$, from one transition epoch to the next, become:
\[
P^{d}[(i,j,t),(i,j,t)'] =
\begin{cases}
\lambda_{1}B^{-1}, & (i,j,t)'=(i+1,j,t+s), \:\:\:\:\:d=WSN\\
\lambda_{1}B^{-1},  & (i,j,t)'=(i,j,t+s), \:\:\:\:\:\:\:\:\:\:\:\: d=DB\\
\lambda_{2}B^{-1},  & (i,j,t)'=(i,j+1,t+s)\\
\mu B^{-1} \phi_{1}(i,j), &(i,j,t)'=(i-1,j,t+s), i>0\\
\mu B^{-1} \phi_{2}(i,j), &(i,j,t)'=(i,j-1,0), j>0 \\
1-(\lambda_{1}+\lambda_{2}+\mu \textbf{1}_{ i+j>0})B^{-1},  &(i,j,t)'=(i,j,t+s)\\
0, &\textrm{otherwise}
\end{cases}
\]
%The next theorem formally shows that the Discrete Time Markov Decision Process and the Continuous Time Markov Decision Process are stochastically equivalent. {\color{red} As the actions are only taken upon arrivals, which occur at exponential times.... this means that the optimal policy determined for the exponentially uniformized MDP will be the same for the CTMDP.}
%
%{\color{red}Let $\pi$ be the policy according to which the controller assigns an incoming query either to the DB or WSN.
%it can be also expected that an optimal assignment policy of the discrete time process  holds for the continuous time process. For more technical justification, see below.}
\begin{thm}\label{lemma}
For any policy $\pi$, the  exponentially uniformized Markov Decision Process and the original Continuous Time Markov Decision Process with a drift are stochastically equivalent.
\end{thm}

\begin{proof}
\ref{proofLemma}
\end{proof}
One consequence of Theorem  is that the expected assignment cost for the exponentially uniformized MDP and the CTMDP with a drift are the same. This, in turn, leads to the same optimal policy for the two processes

Now observe that in the CTMDP with a drift, the actions are only taken upon query arrivals, which occur at exponential times. In the case of the exponentially uniformized MDP, the exponential times have parameter $B$.
Thus, the actions will still be taken at exponential times with parameter $B$, upon a query arrival. 
Therefore, it is sufficient to keep track of the number of exponential phases $N$  (Erlang distribution with parameter $B$ and N phases).
This allows us to restrict ourselves to a discrete time and space Markov decision process in section \ref{Discrete Time Markov Decision Process}. A discrete time and space MDP enables us to compute an optimal assignment policy in section \ref{Numerical results}.
%Keeping track, however, of the number of exponential phases with parameter $B$, instead of the continuous component $t$ of the exponentially uniformized MDP, will not change the optimal assignment policy.

In the next section, therefore, we restrict ourselves to a discrete space and time Markov Decision Problem, with $S= \mathbb{N}_{0} \times \mathbb{N}_{0} \times \mathbb{N}_{0}$, where $(i,j,N) \in S$ denotes the state in which there are $i$ queries, $j$ reports and $N$ steps since last report completion, i.e. the age of the data is given by the number of exponential phases $N$ . 

\subsubsection{Discrete Time and Space Markov Decision Problem}\label{Discrete Time Markov Decision Process}

Based on the  exponentially uniformized model in Section \ref{Exponential Markov Decision Process}, we formulate our assignment problem as a Discrete Time and Space Markov Decision Problem (DTMDP) as follows:
\begin{itemize}
\item State space: $S= \mathbb{N}_{0} \times \mathbb{N}_{0} \times \mathbb{N}_{0} $, where $(i,j, N) \in S$ denotes the state with $i$ queries and $j$ reports in the WSN and $N$ is the age of the stored data, with $N$ the number of steps (exponentially distributed with uniformization parameter $B$) since the last report completion. 
\item Action space: 
Upon a query arrival, the controller takes an action $d$ from the action space $D=\{DB, WSN\}$, where $d=DB$ is a DB assignment and $d=WSN$ is a WSN assignment. 
\item Transition probabilities, when the system is in state $(i,j,N) \in S$ and action $d \in D$ is taken, are as follows:

\begin{align} \label{TransitionProbabilities}
P^{d}[(i,j,N),(i,j,N)'] =
\begin{cases}
\lambda_{1}', & (i,j,N)'=(i+1,j,N+1), \:\:\:\:\:d=WSN\\
\lambda_{1}',  & (i,j,N)'=(i,j,N+1), \:\:\:\:\:\:\:\:\:\:\:\: d=DB\\
\lambda_{2}',  & (i,j,N)'=(i,j+1,N+1) \\
\mu' \phi_{1}(i,j), &(i,j,N)'=(i-1,j,N+1), i>0 \\
\mu' \phi_{2}(i,j), &(i,j,N)'=(i,j-1,0), j>0 \\
1-(\lambda_{1}'+\lambda_{2}'+\mu' \textbf{1}_{ i+j>0}),  &(i,j,N)'=(i,j,N+1) \\
0, &\textrm{otherwise}
\end{cases}
\end{align}
with $\phi_{1}(i,j)=\frac{i}{i+j}$, $\phi_{2}(i,j)=\frac{j}{i+j}$ and $\lambda_{i}'=\lambda_{i}B^{-1}, i\in\{1,2\}$ and $\mu'=\mu B^{-1}$ as per uniformization (see subsection \ref{Exponential Markov Decision Process}).
The first two lines of (\ref{TransitionProbabilities}) model query arrivals under action $d$. The third line  of (\ref{TransitionProbabilities}) models report arrivals. The fourth and fifth lines  of (\ref{TransitionProbabilities}) model query and report completions, respectively. The sixth line  of (\ref{TransitionProbabilities}) is a dummy transition as a result of the uniformization. The last line  of (\ref{TransitionProbabilities}) prohibits any other state transition. Notice that in every step, the age is incremented, except the case when a report is completed. Then, the age is reset to zero.

\item Cost function: The cost of the system is two-folded. Firstly, when $i$ queries are waiting to be solved within the WSN, the system incurs a cost  per unit of time:
\begin{equation}\label{costW}
i
\end{equation} This can be interpreted as, each unit of time, the system pays one unit for each waiting query. At the end of a query's service, the system would have payed one unit for each unit of time the query was in the system, i.e. the query waiting time. 
Secondly, if an incoming query is assigned to the DB, an instantaneous  penalty  is incurred for exceeding the validity tolerance $T$  of the stored data:
\begin{equation} \label{costD}
\max (N'-T)^{+}, \:\: (x)^{+}=\max\{0,x\}.
\end{equation} where $N'=N/B$ is the age of the data in time units, i.e. the number of uniformization steps multiplied by the expected length of a step. 
In this case, the system pays for the time the data validity is exceeded.
Considering the cost of having queries waiting in the WSN (\ref{costW}) and the instantaneous cost associated with a DB assignment (\ref{costD}), when the system is in state $(i,j,N)$, the cost incurred per unit of time is:
\begin{equation}
C(i,j,N)=i+\lambda_1 (N'-T)^+\textbf{1}_{(d=DB)}, \textrm{ where} (x)^{+}=\max\{0,x\}.
\end{equation} 
\end{itemize}

\textbf{Remark}: The number of exponential phases approximates the time until a report completion by $t+s=(N+1) \cdot B^{-1}$. 
Also, the variance of an Erlang distribution with $N$ phases and parameter $B$, which is the case for our discretized age, is $\frac{N+1}{B^2}$. 
As $B \geq \lambda_1+\lambda_2+\mu$ can be chosen arbitrarily large (see \cite{hordijk1982weak}), by the law of large numbers, for very large $B$, the distribution of Erlang(N+1,B) will concentrate around $(N+1) \cdot B^{-1}$.
Thus, for large uniformization parameter $B$, the discrete time and state MDP approximates the uniformized MDP arbitrarily close. 
%Thus, it is expected that, for sufficiently large B, the optimal policy will not change. 

On expectational basis, the value of the uniformization parameter $B \leq \lambda_1+\lambda_2+\mu$ can be seen as a scaling factor that does not influence the results. Several examples have been investigated in \ref{B} also showing no effect of $B$ on the assignment policy. 
One could expect that for small values of $B$, a minor effect on the policy might be present due to the approximation of the age component $N'=N/B$ (see (\ref{costD})). However, we have not been able to find any such example. In other words, the approach followed is  strongly supported, both theoretically and numerically.

Now, the quadruple $(S,D,P,C)$ completely describes the discrete time and state MDP.
\\
To determine an optimal assignment policy and to use standard dynamic programming, we define the following value function:
\begin{equation*}
\textbf{V}_{n}(i,j,N) :=\textrm{minimal expected assignment cost over } n \textrm{ steps starting in state } (i,j,N).
\end{equation*}
Then $\textbf{V}_{n}(i,j,N)$ is computed recursively by means of the value iteration algorithm (see, for instance,  \cite{puterman1994markov}  Section 8.5.1 ) as follows:

First, we consider $\textbf{V}_{0}(i,j,N)=0$. Next, we iterate according to the  value iteration algorithm and the following backward recursive equation:

\begin{equation}
\label{eq:backward} \textbf{V}_{n+1}(i, j, N) =
\begin{cases}
                                           i' + \lambda_{1}' \min
							\begin{cases}
							V_{n}(i+1, j, N+1)\\
 							(N-T')^{+}+V_{n}(i, j, N+1)
							\end{cases}  \\
				+\lambda_{2}' V_{n} (i, j+1, N+1)    \\
				+ \mu' \phi_{1}(i,j) V_{n} (i-1, j, N+1) \textbf{1}_{i>0}    \\
				+  \mu' \phi_{2}(i,j) V_{n} (i, j-1, 0)  \textbf{1}_{j>0}   \\
				+ [1-(\lambda_{1}'+\lambda_{2}'+\mu'\textbf{1}_{i+j>0})] V_{n}(i, j, N+1).   
\end{cases}
\end{equation} where $i'=i/B$ and $T'=T/B$, following uniformization.
The first term of \eqref{eq:backward} is the cost of having $i$ queries in service and  a query assigned to either the WSN or the DB.
The next three terms represent the cost incurred by a transition due to a report arrival, a query completion and a report completion, respectively. 
Lastly, the final term is the dummy term due to uniformization.

Simultaneously with computing $\textbf{V}_{n}(i,j,N)$, the algorithm computes a  $\epsilon$-optimal stationary policy $\pi_{n}$ which associates an optimizing action with the right-hand side of ~\eqref{eq:backward} for any state $(i,j,N)$. 
Given the assignment policy, it is possible to compute the average assignment cost. 

Denote the minimal average assignment cost by $g^{*}$.
%\begin{equation}\label{g*}
%g^{*}= \lim_{ n \to \infty} [V_{n+1}(i,j,N)-V_{n}(i,j,N)].
%\end{equation}
Since the underlying Markov chain is ergodic, $g^{*}$ is independent of the initial state. 
We approximate $g^{*}$ using the following bounds introduced in \cite{odoni1969finding}:
\begin{align} \label{bounds}
L'_{n} &\leq g^{*} \leq L''_{n}, \textrm{  where }  \\ %
 L'_{n} &=\min[V_{n+1}(i,j,N)-V_{n}(i,j,N)],  \notag \\
L''_{n} &=\max[V_{n+1}(i,j,N)-V_{n}(i,j,N)]. \notag
\end{align}
In (\ref{bounds}), $L'_{n}$ is the minimum difference of  the value function over two iteration steps, $n$ and $n+1$, whereas $L''_{n}$ is the maximum difference of the value function over steps $n$ and $n+1$. For $n\rightarrow \infty$, $L'_{n}$ and $L''_{n}$ become arbitrarily close.

The optimal cost $g^{*}$ is computed with an accuracy $\epsilon$ by iterating the right-hand side of \eqref{eq:backward} for $n$ times until  
$L"_{n}- L'_{n} \leq \epsilon/B$ with $B$ the uniformization parameter.
Then, the average assignment cost is approximated as $\displaystyle g^{*} \sim \frac{(L"_{n}+L'_{n})}{2}.$
It can be shown that the lower and upper bound converge in a finite number of steps (Theorem 8.5.4 \cite{puterman1994markov} ) to the optimal cost. 

\section{Numerical Results}\label{Numerical results}

\subsection{Numerical Results - Optimal Query Assignment Policy} \label{Numerical Results - Optimal Query Assignment Policy}
Based on the Discrete Time and State Markov Decision Process defined in subsection \ref{Discrete Time Markov Decision Process}, we were able to compute an optimal query assignment policy.
%\begin{figure}[ht]
% \centering 
%
% \subfigure[T=1, N=30]{  
%\includegraphics[height=5cm]{T1N30.eps}
%}  
% \subfigure[T=4, N=30]{  
%
%\includegraphics[height=5cm]{T4N30.eps}
% }  
%\subfigure[T=1, N=50]{ 
% 
%\includegraphics[height=5cm]{T1N50.eps}
%}  
% \subfigure[T=4, N=50]{  
%
%\includegraphics[height=5cm]{T4N50.eps}
% }  
%\caption{WSN assignment (black) and DB assignment (grey) assuming $\lambda_{1}=0.8$, $ \lambda_{2}=0.5$ and $\mu=1.8$. $N$ is the age of the data and $T$ the validity threshold.}
%\label{optimal}
% \end{figure}  
\begin{figure}[h]
\centering
\includegraphics[scale=1]{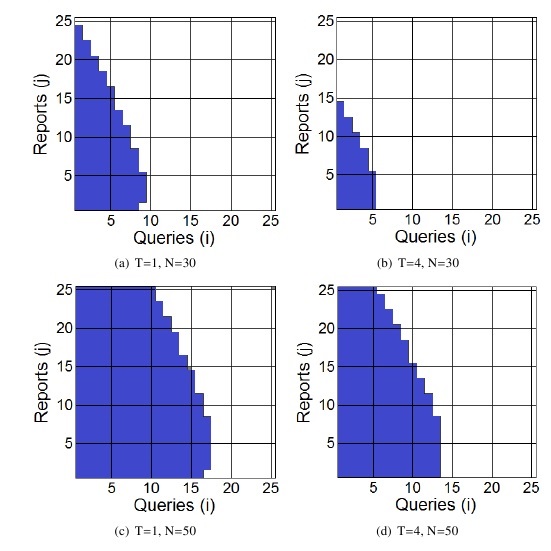}
\caption{WSN assignment (blue) and DB assignment (white) with $\lambda_{1}=0.8$, $ \lambda_{2}=0.5$ and $\mu=1.8$. $N$ is the age of the data and $T$ the validity threshold.}
\label{optimal}
\end{figure}
Figure \ref{optimal} shows what action is optimal when the system is in state $(i,j,N)$ and data validity tolerance $T$ is assumed.
%	{\color{red}We remark here that the choice of the uniformization parameter $B> \lambda_1+\lambda_2+\mu$ does not affect the optimal policy. Moreover, taking $B^{-1} \rightarrow 0$, the discrete time approximation is guaranteed. This illustrates again the fact that the discrete time and space MDP as described in \ref{Discrete Time Markov Decision Process} and the continuous time MDP described in \ref{Continuous Time Markov Decision Process} are stochastically equivalent and the optimal policy derived using the discrete MDP also holds for the continuous time MDP.}

\subsection{Fixed Heuristics Policies for Performance Comparison} \label{Fixed Heuristics Policies for Performance Comparison}

In practice, simple assignment policies are employed to manage the query traffic. We consider the following three assignment heuristics, derived from practical assignment strategies:
\begin{itemize}
\item A fixed heuristic policy $\pi^{Db}$ that always assigns incoming queries to the DB. 
Upon a query arrival, the cost incurred is $(N-T)^{+}$. 
\item A fixed heuristic policy $\pi^{W}$ that always assigns incoming queries to the WSN.  
\item A heuristic policy $\pi^{T}$ that always assigns incoming queries to the DB  if the age does not exceed the tolerance, i.e. $N<=T$, and to the WSN otherwise. 
\end{itemize}

The following theorem shows what are the expected assignment costs incurred by the  $\pi^{Db}$ and $\pi^{W}$ heuristics when data validity threshold $T$ is assumed.
\begin{thm}\label{lemma2}
Assuming the DTMDP parameters $\lambda_{1}',\lambda_{2}' \textrm{ and } \mu'$, the average assignment cost of the heuristics $\pi^{Db}$ and $\pi^{W}$ are as follows,
\begin{equation}\label{costDB} 
C_{\pi^{Db}}=\frac{ \lambda_{1}' (1-\lambda_{2}')^{T+1}  }{\lambda_{2}'}
\end{equation}

\begin{equation}\label{costWSN}
C_{\pi^{W}}=\frac{\lambda_{1}'}{\mu'-(\lambda_{1}'+\lambda_{2}')}
\end{equation}

 \end{thm}
 
\begin{proof}\ref{CostWSN} \end{proof}

\subsection{Simulation results}

We compare the performance of the proposed assignment policy, i.e.  the associated average cost ($g^{*}$), as defined in Section \ref{Model Formulation}, with the average assignment cost of the heuristics proposed in subsection \ref{Fixed Heuristics Policies for Performance Comparison} by means of a discrete event simulation. 
While for the  $\pi^{Db}$ and $\pi^{W}$ heuristics exact results are derived in Theorem \ref{lemma2}, we use simulation to compute the average assignment costs for heuristic $\pi^T$. Moreover, we use simulation as we are interested in the fraction of time the DB or the WSN are used. This gives us an indication of the load of the WSN over time. 

Simulation results show that, compared with th heuristics, the proposed policy, described in subsection \ref{Discrete Time Markov Decision Process}, achieves a lower average assignment cost  (Figure \ref{aaa}).  The cost difference is significant for small time tolerances. This is of particular interest for real-time applications which specify low time tolerances. 
At the limit, $T \rightarrow \infty$, both $\pi^{T}$ and $\pi^{Db}$ approach the optimal policy. In this case, the stored data is considered valid for a long time. Consequently, DB assignments under $\pi^{T}$, $\pi^{Db}$ and $\pi^{OPT}$ become more frequent under all these policies  (Figure \ref{aaa4}) and the costs converge to the cost of the optimal policy. In short, for large validity tolerances, it is always optimal to send incoming queries to the DB.
\begin{figure}[ht]\label{f}
 \centering 

 \subfigure[Average Assignment Cost]
 {  
\includegraphics[scale=0.6]{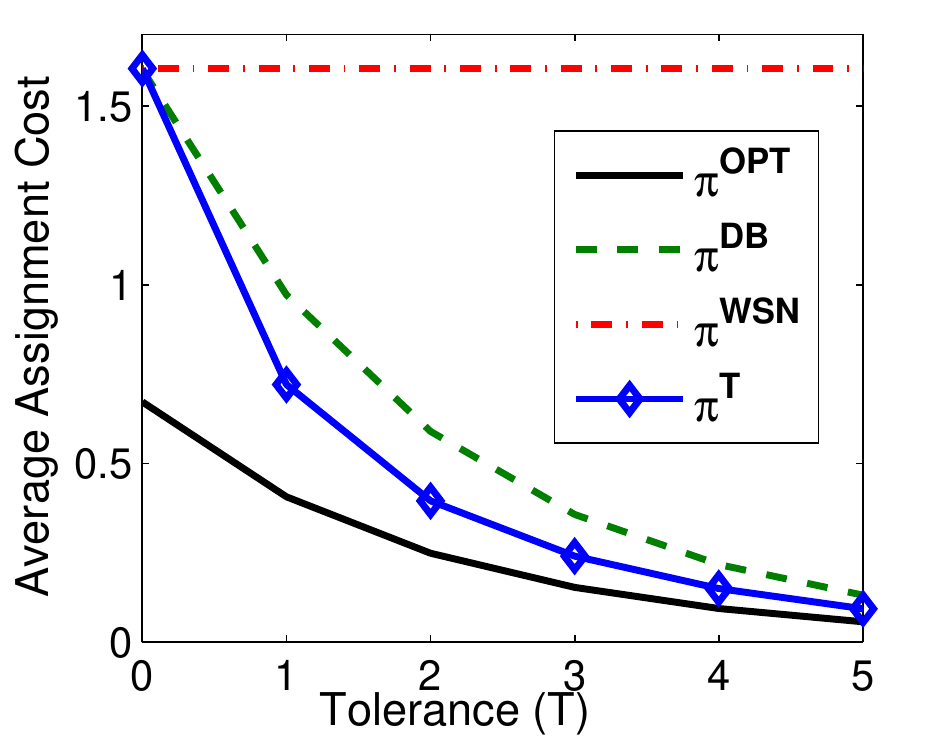}\label{aaa} %Policies.eps
 }  
 \subfigure[DB utilization]
 {   
\includegraphics[scale=0.6]{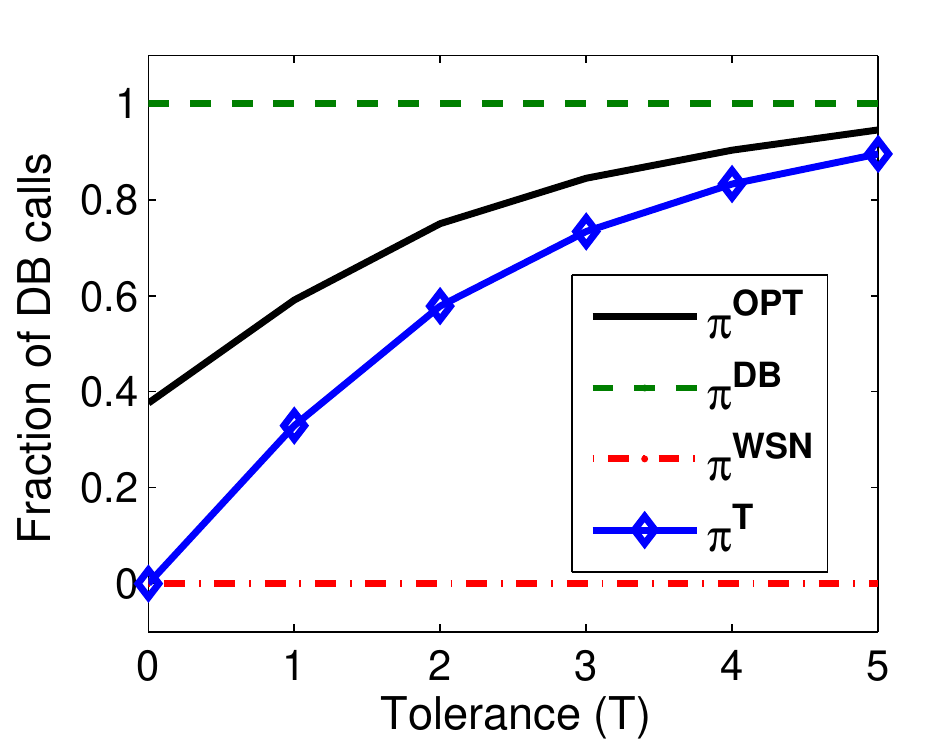}\label{aaa4} %4Policies-dbutil.eps
 }  
\caption{ Average Assignment Cost and DB utilization assuming query and report arrival rates $\lambda_{1}=0.8$ and $ \lambda_{2}=0.5$, respectively and WSN service rate $\mu=1.8$}
 \end{figure} 

Also in the case of increasing query arrival rate (Figures \ref{aaa2} and \ref{aaa3}) or WSN processing capabilities, the optimal policy outperforms the heuristic policies in terms of average assignment costs (Figures \ref{mu1} and \ref{mu2}). 
\begin{figure}[ht]
 \centering 
 \subfigure[Average Assignment Cost]
 {  
\includegraphics[scale=0.6]{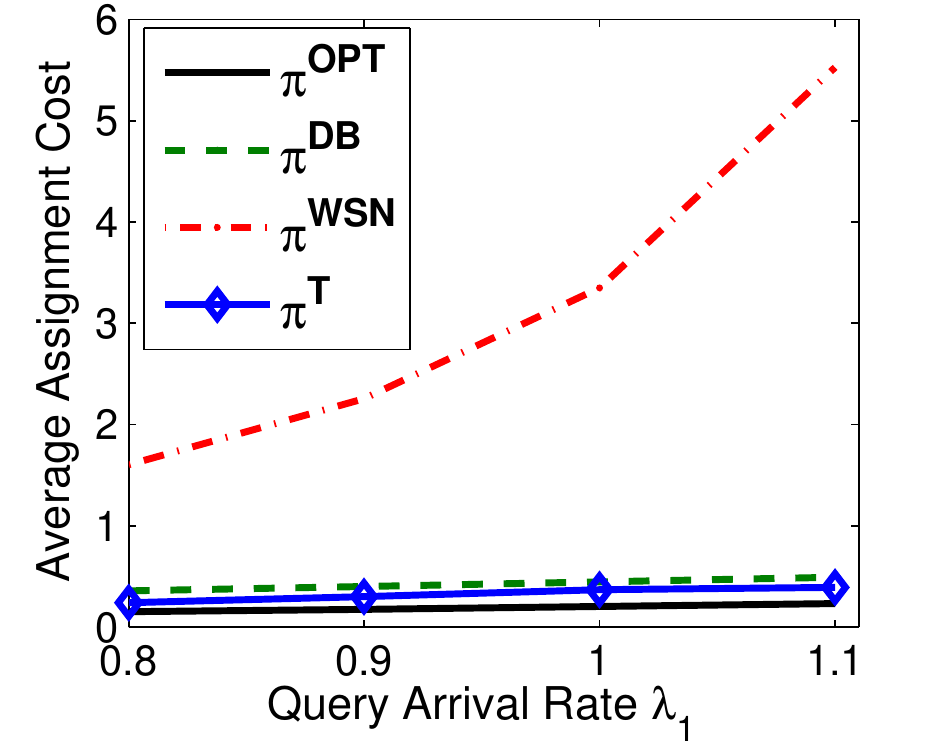}\label{aaa2} %5a->Arrivals
 }  
 \subfigure[Average Assignment Cost - Zoom in]
 {   
\includegraphics[scale=0.6]{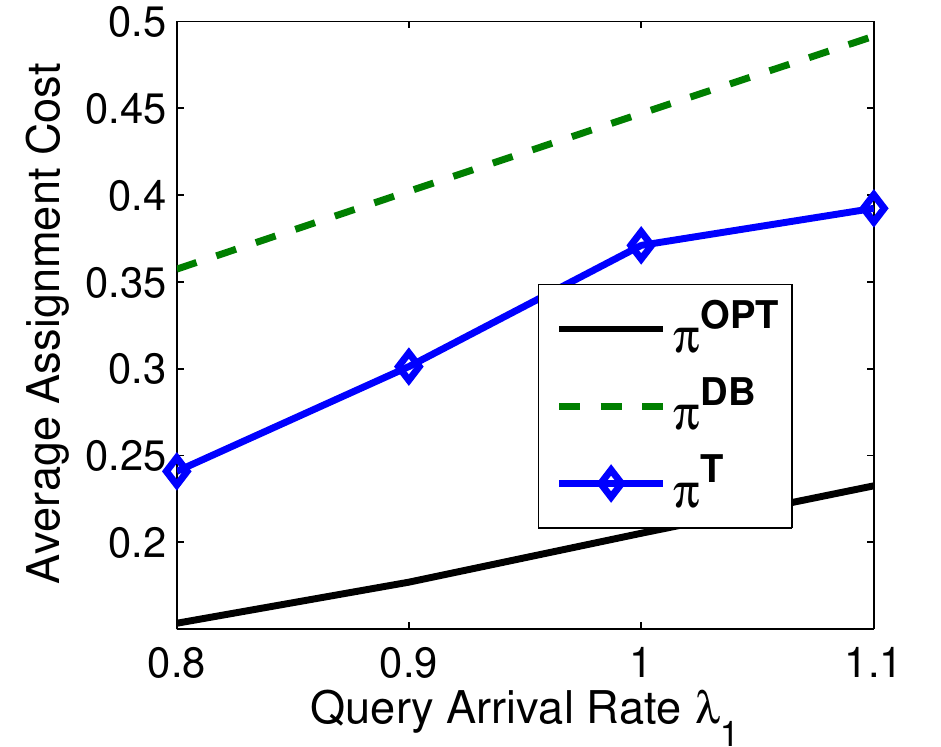}\label{aaa3} %5b->Arrivals2
 }  
\caption{ Average Assignment Cost for different query arrival rates $\lambda_{1}$, $ \lambda_{2}=0.5$, $\mu=1.8$ and $T=1$}
 \end{figure} 

\begin{figure}[ht]
 \centering 
 \subfigure[Average Assignment Cost]
 {  
\includegraphics[scale=0.6]{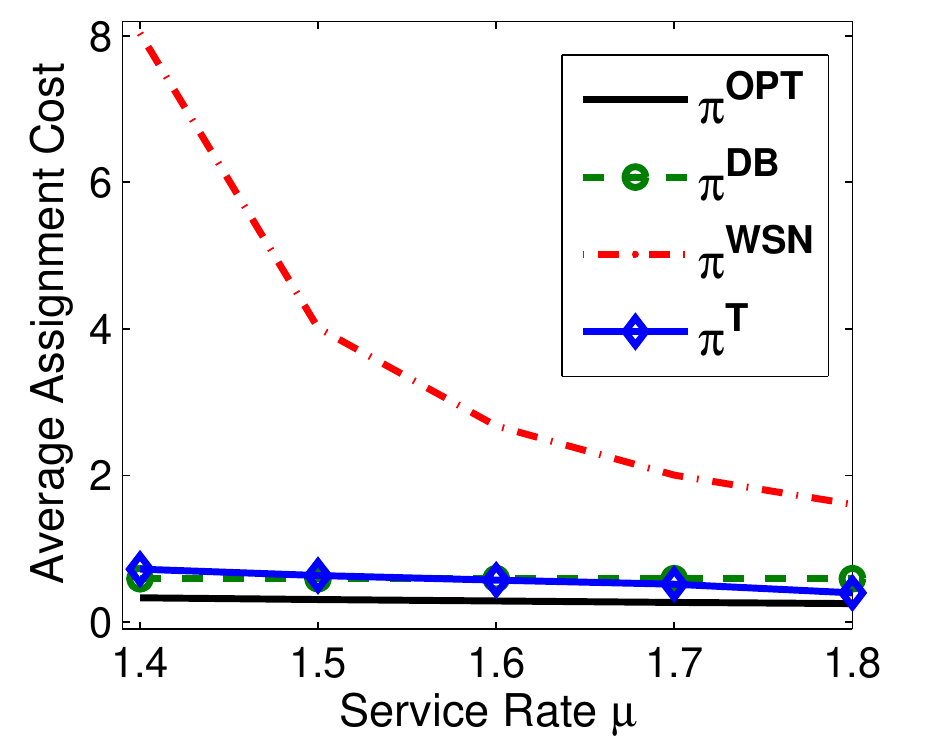}\label{mu1} %6a->varying-mu
 }  
 \subfigure[Average Assignment Cost - Zoom in]
 {   
\includegraphics[scale=0.6]{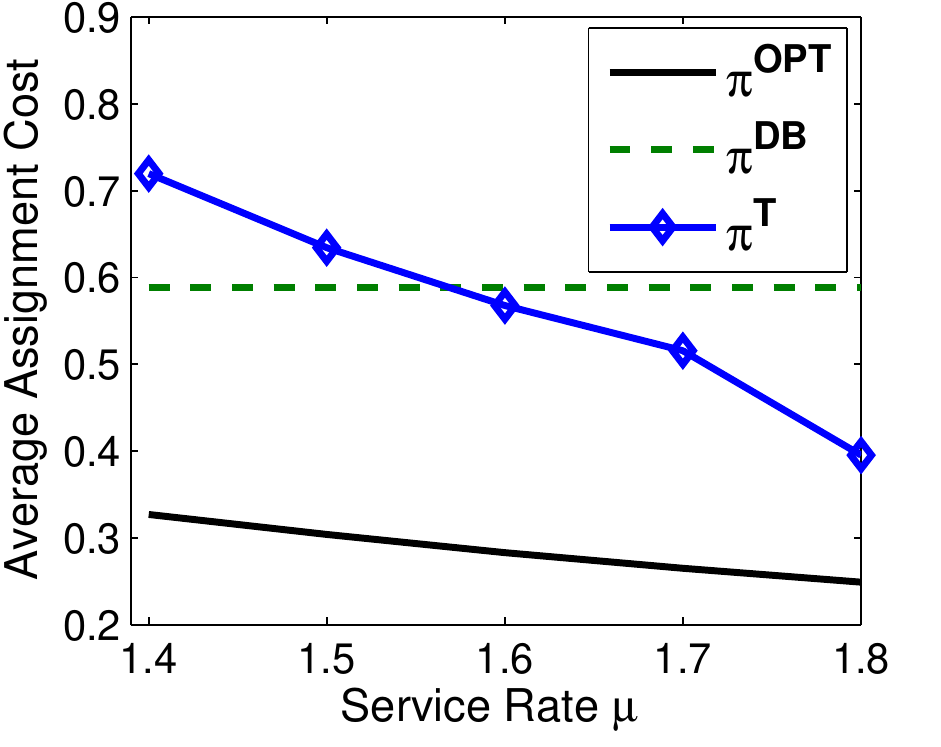}\label{mu2} %6b->varying-mu-zoom
 }  
\caption{ Average Assignment Cost for different processing capabilities $\mu$, $ \lambda_{1}=0.8$, $\lambda_{2}=0.5$ and $T=1$}
 \end{figure} 
Such insight into the performance of the system enables WSN service providers to deliver customized and efficient monitoring services to end-users.
For reasonably large data validity tolerances, simple heuristics such as $\pi^{Db}$ or $\pi^{T}$ perform well in comparison to the optimal policy. These heuristics are particularly suitable for monitoring environments with little variation over time, e.g. temperature sensing.
However, for applications with highly constrained delivery requirements and large data variance over time, such as fire detection or $CO_{2}$ monitoring, the validity tolerance $T$ is expected to be low. In this case, our proposed model outperforms the heuristics.
Moreover, as seen in Figure \ref{optimal}, the optimal policy assigns incoming queries to the WSN only if the number of reports in service exceeds the number of queries. A large number of reports in service ensures frequent DB updates which, in turn, decreases the assignment costs.

\subsection{Policy Simulations for Real Query Traffic} \label{Results with real query arrivals}

In this subsection, we assess the performance of the  above described policies using data obtained from a commercial sensor network platform \cite{Munisense}.
We use a logfile containing timestamps (in seconds) of the queries arriving at the platform. We selected two time periods, shown in Figures \ref{fig:dataset1} and \ref{fig:dataset2}, which are representative of the intensity of query arrivals.

\textbf{Dataset 1} contains timestamps of queries from one weekday around lunchtime, when the platform typically receives many queries (Figure \ref{fig:dataset1}). The number of queries arriving at the platform per minute varies from about 25 in a busy period, to 5 or fewer during a quiet period. The coexistence of such busy periods and quiet periods makes the assumption of query arrivals originating from a homogeneous Poisson process invalid. Hence, verifying the performance of our optimal policy on such non-Poisson data provides valuable insight into its practical relevance.

\textbf{Dataset 2} has timestamps of queries arriving at the platform during night-time (Figure \ref{fig:dataset2}). In this timespan, most queries are generated by a periodically refreshing dashboard and show little variance.
\begin{figure}[ht]
	\centering 
	\subfigure[Dataset 1]{  
		\includegraphics[scale=0.4]{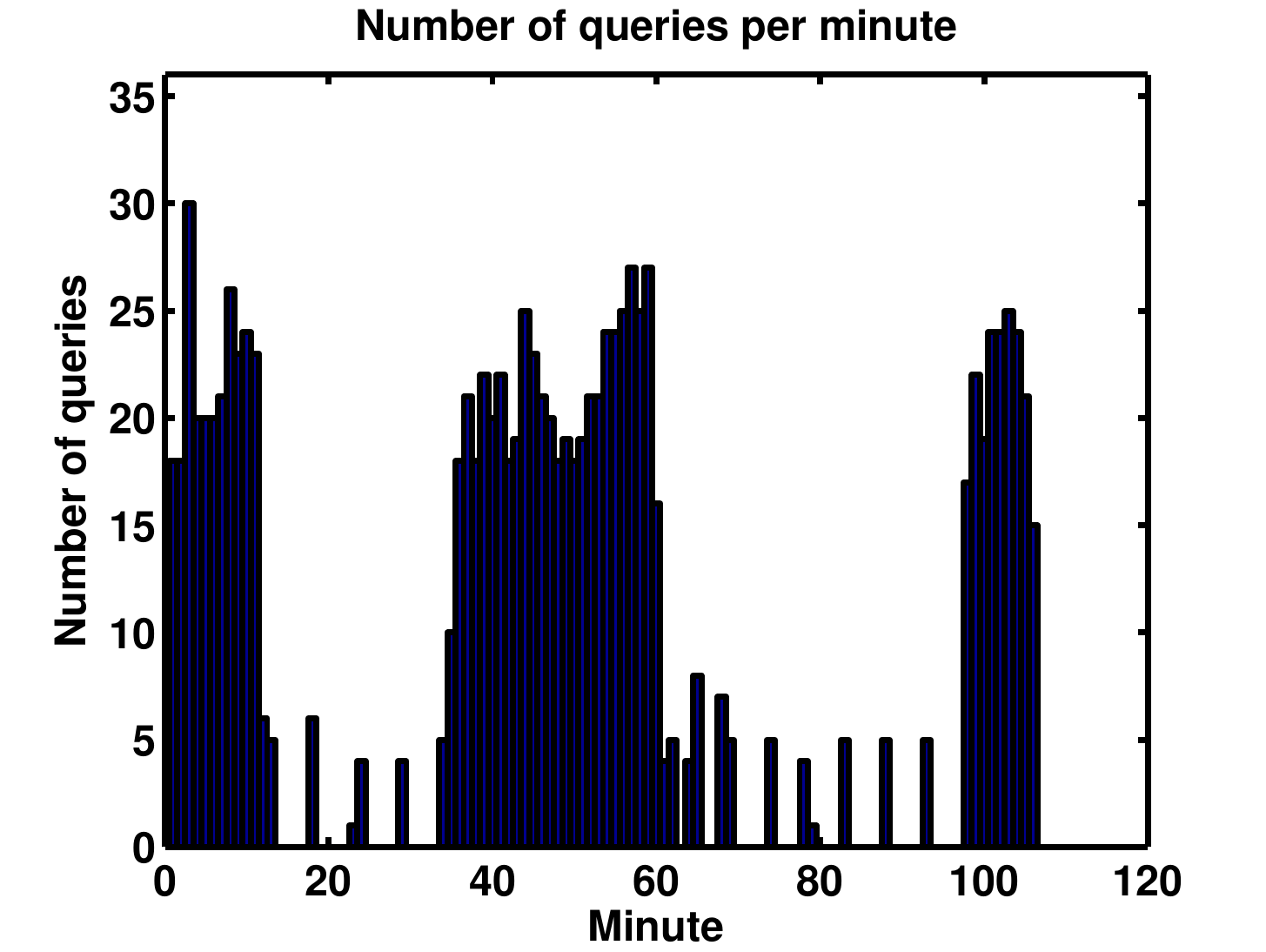}\label{fig:dataset1} %7a->dataset1
	}  
	\subfigure[Dataset 2]{  
		\includegraphics[scale=0.4]{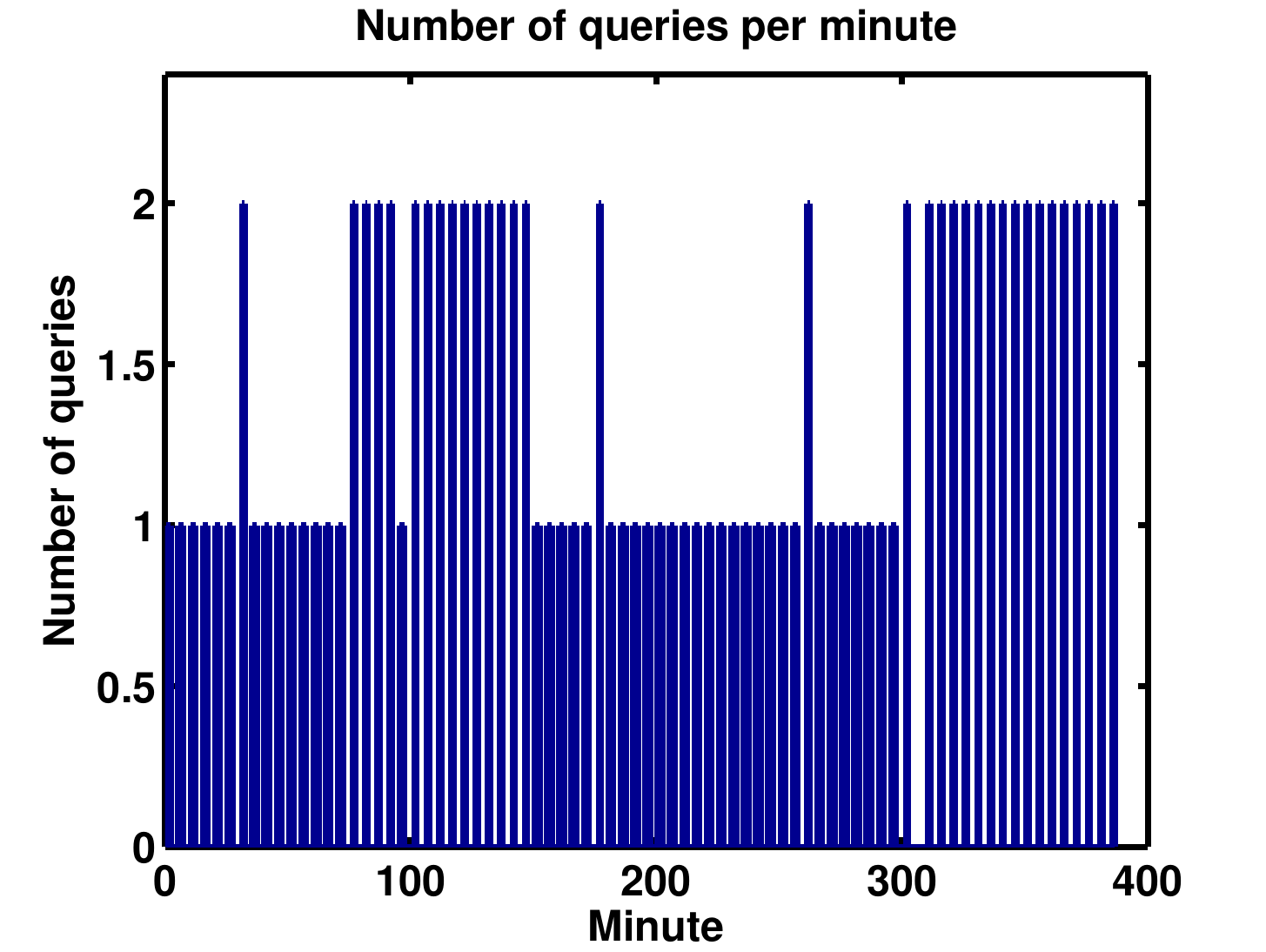}\label{fig:dataset2} %7b->dataset2
	}
	\caption{The number of arriving queries per minute for the two datasets.}
\end{figure}  

We perform a discrete event simulation and use the timestamps from the datasets as the arrival times of the queries. The optimal policy is determined using the procedure outlined in Section \ref{Model Formulation}. The query arrival rate, $\lambda_1$,  is estimated from the mean inter-arrival time of the queries in the datasets. We choose the report arrival rate $\lambda_2$ and the service rate $\mu$ such that the system has the same load as the one in Figure \ref{aaa}.
\begin{figure}[ht]
	\centering 
	\subfigure[Average Assignment costs]{  
		\includegraphics[scale=0.6]{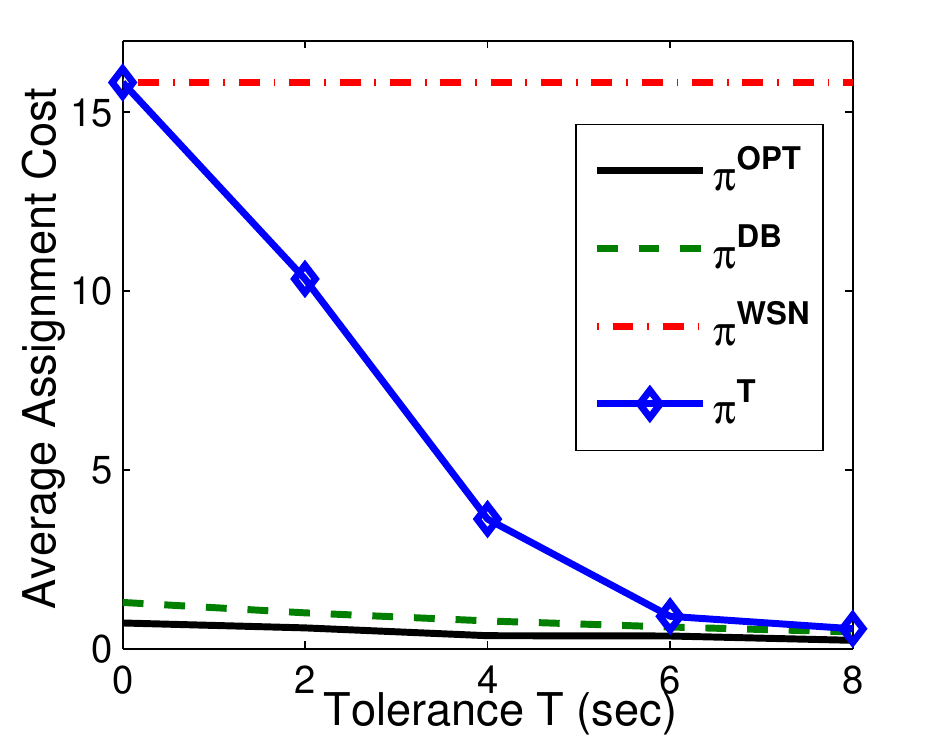}\label{fig:dataset1-costs} %8a->costs-dataset1
	}  
	\subfigure[DB utilization]{  
		\includegraphics[scale=0.6]{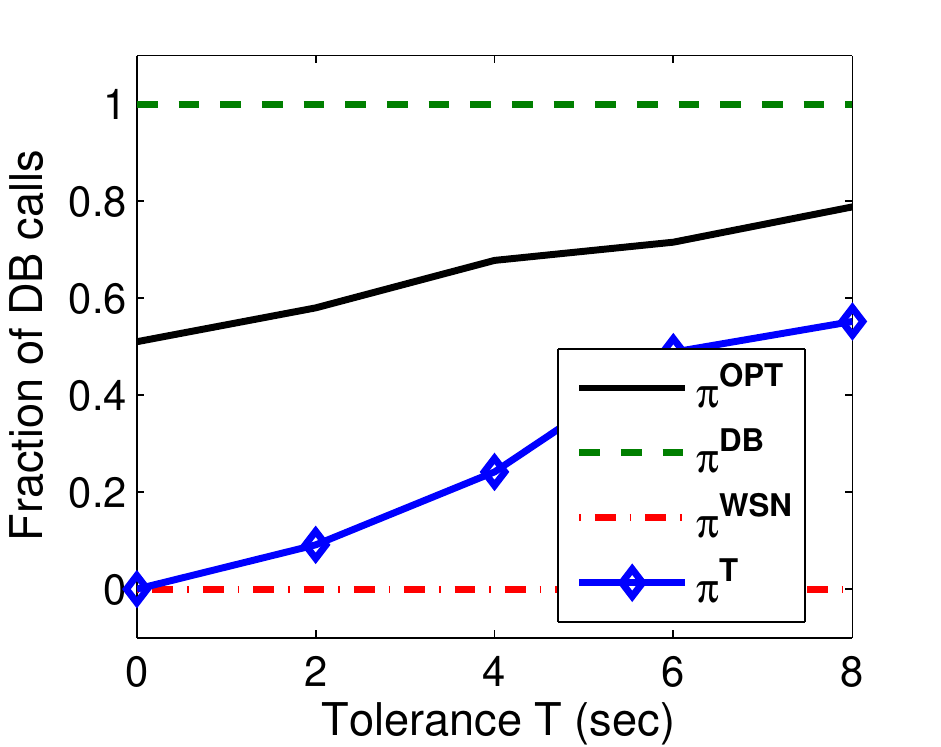}\label{fig:dataset1-dbutil} %8b->dbutil-dataset1
	}  
	\caption{Average Assignment costs and DB utilization for Dataset 1}
\end{figure} 

For Dataset 1, Figures \ref{fig:dataset1-costs} and \ref{fig:dataset1-dbutil} show that our optimal policy outperforms the heuristics in terms of the average assignment costs. The load of the WSN is considerably decreased by routing the queries to the DB. The difference in performance is especially visible for smaller time tolerances, where the optimal policy achieves lower average costs while making more use of the DB.
Results are similar for Dataset 2 (Figures \ref{fig:dataset2-costs} and \ref{fig:dataset2-dbutil}). %The optimal policy performs better than the fixed heuristics in terms of average assignment costs.
\begin{figure}[ht]
	\centering 
	\subfigure[Average Assignment costs]{  
		\includegraphics[scale=0.6]{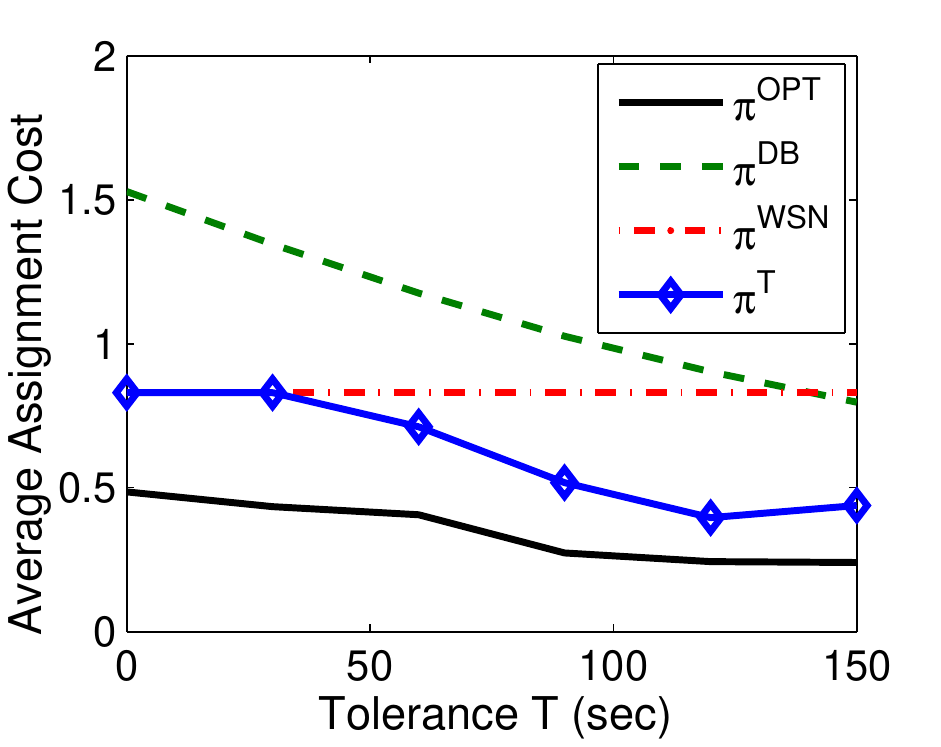}\label{fig:dataset2-costs} %9a->costs-dataset2
	}  
	\subfigure[DB utilization]{  
		\includegraphics[scale=0.6]{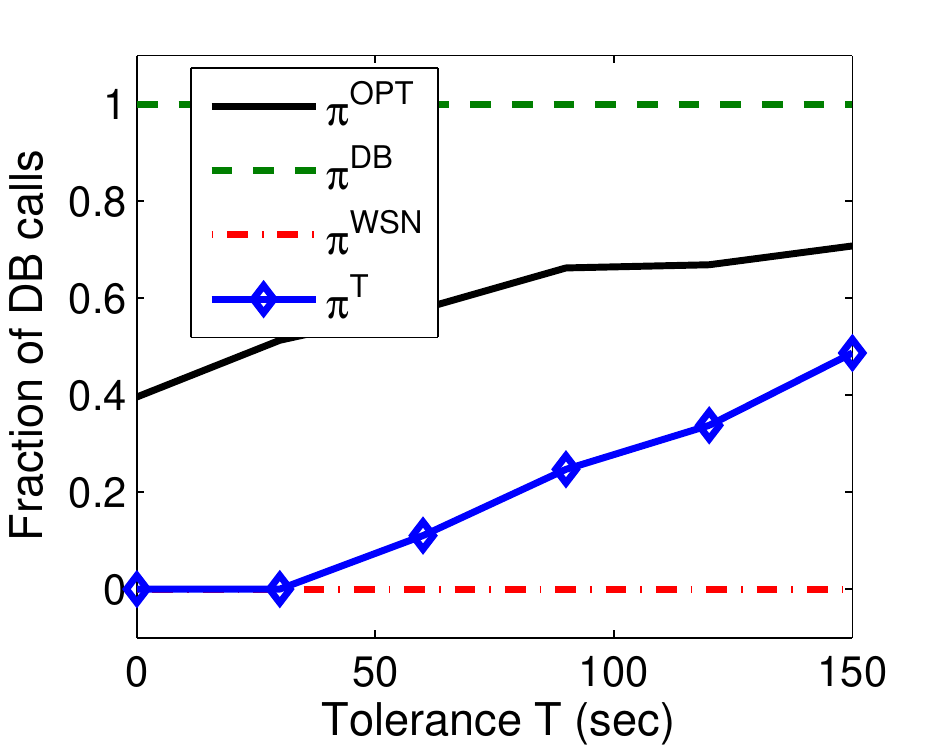}\label{fig:dataset2-dbutil} %9b->dbutil-dataset2
	}  
	\caption{Average Assignment costs and DB utilization for Dataset 2}
\end{figure} 
Simulation results show that the optimal policy achieves cost savings that are independent of the assumption that the arrivals at the platform follow a Poisson process. 

These simulation results emphasize the practical applicability of our proposed model. Lastly, we point out that the proposed  model is independent of the wireless sensor platform used and its applications.

\section {Conclusion and Future work}\label{Conclusion and Future work}

This paper investigated the trade-off between the query waiting  time and the validity of the stored data.
Firstly, we defined the query assignment problem as a Continuous Time Markov Decision Process with a drift. We next defined, for computational purposes, a stochastically equivalent,  uniformized Markov Decision Problem. We provided an optimal query assignment strategy and assessed numerically its performance. We showed that for low validity tolerance, the proposed policy achieves significant cost saving in comparison to several feasible heuristics, commonly used in practice. Lastly, we showed that the proposed assignment policy outperforms the heuristics also in the case of real-life query traffic.
%The proposed assignment strategy also outperforms the heuristics in the case of a representative WSN platform. 
%Future work includes extending the query assignment model to incorporate additional metrics such as reporting rate, reliability of the sensors and energy efficiency. 

\section*{Acknowledgements:} This work was performed within the project RRR (Realisation of Reliable and Secure Residential Sensor Platforms) of the Dutch program \emph{IOP Generieke Communicatie}, number IGC1020, supported by the \emph{Subsidieregeling Sterktes in Innovatie}.

%% The Appendices part is started with the command \appendix;
%% appendix sections are then done as normal sections

 \appendix

\section{ Proof of Theorem \ref{lemma}}
\label{proofLemma}

\begin{proof}
Uniformization is commonly used for  Markov jump processes, making the problem computationally tractable. As a drift component is introduced in the present setting (the age component of our process evolves continuously in time), this is no longer standard.

The infinitesimal generators uniquely define a Markov process. Therefore, it is sufficient to show that the infinitesimal generators of the exponential uniformized Markov Decision Process and the original Continuous Time Markov Decision Process with a drift are identical.

To prove this, let $P^{d}_{\Delta t}$ denote the transition probability measures over time interval of length $\Delta t$, given that at the last jump the system is in state $(i,j,t)$ and that following a next jump, decision $d$ is taken. We implicitly assume that a policy $\pi$, prescribing an action $d$ upon a query arrival when the system is in state $(i,j,t)$, is left continuous.

Let $f:\mathbb{N} \times \mathbb{N} \times \mathbb{R}$ be an arbitrary real valued function, differentiable in $t$ and $o (\Delta t)^{2} \leq C o(\Delta t)^{2}$ for any constant $C$.
Then by conditioning upon the exponential jump epoch with variable $B$ and for arbitrary $f$ we obtain,
\begin{align*}
P^{d}_{\Delta t} f(i,j,t) = &\: e^{-\Delta t \cdot B}  f(i,j,t+ \Delta t)+
				 \int_{0}^{\Delta t} B e^{-sB} \sum \limits_{(i,j,t)'} P^{d}[(i,j,t),(i',j',t+s) ] f (i',j', t+s) ds+o(\Delta t)^{2} \\
				&=f(i,j,t+\Delta t)-
				\Delta t B f(i,j,t+\Delta t)+
				\Delta t B \sum \limits_{(i',j') \neq(i,j)} q^{d}[(i,j,t),(i',j',t)]f(i',j', t+ \Delta t) B^{-1}\\
				&+ \Delta t B [1-q^{d}(i,j) B^{-1}]f(i,j,t+\Delta t) + o(\Delta t)^{2}\\
				&=f(i,j, t+\Delta t) +B \sum \limits_{(i',j') \neq(i,j)}q^{d}[(i,j,t),(i',j',t)] [f(i',j', t+ \Delta t)- f(i,j,t+\Delta t)] +  o(\Delta t)^{2}
\end{align*}
where 
$q^{d}[(i,j,t), (i',j',t)]=q^{d}[(i,j,t+s), (i',j',t+s)]$ for any '$(i',j) \neq (i,j)$ and arbitrary $s$.
The term $o(\Delta t)^{2}$ reflects the probability of at least two jumps and the second term of the Taylor expansion for $e^{-\Delta B}$.

Hence, by subtracting $f(i,j,t)$, dividing by $\Delta t$ and letting $\Delta t \rightarrow 0$, we obtain,
\begin{align*}
\frac{P^{d}_{\Delta t} f(i,j,t)  - f(i,j,t)}{ \Delta t}= &\:[ f(i,j,t+ \Delta t) - f(i,j,t)]/ \Delta t\\
								& + B [ f(i,j,t+ \Delta t) - f(i,j,t)] + o(\Delta t)^{2}\\
								&+ \sum \limits_{(i',j') \neq (i,j)} q^{d}[(i,j,t),(i',j',t)] [f(i',j', t)-f(i,j,t)]\\
								&\rightarrow \frac{d}{dt} f(i,j,t)+ \sum  \limits_{(i',j') \neq (i,j)} q^{d}[(i,j,t),(i',j',t)][f(i',j', t)-f(i,j,t)]\\
								&=\textbf{A}^{d} f(i,j,t) \:\:\: \textrm{which is the generator in (\ref{Agenerator})}.
\end{align*}
Since the exponentially uniformized Markov decision process (as defined in section \ref{Exponential Markov Decision Process}) and the continuous time Markov decision process with a drift (defined in section \ref{Continuous Time Markov Decision Process}) share the same generators (see \cite{dynkin1965markov}), the two processes are stochastically equivalent.
\end{proof}

\section{ Proof of Theorem \ref{lemma2}}
\label{CostWSN}

\begin{proof}
We first analyze the expected assignment cost under the policy $\pi^{W}$.

The $\pi^{W}$ policy is independent of the validity tolerance. The WSN behaves as a regular M/M/1 Processor Sharing queue. Therefore, the cost of the heuristic is given by the expected number of jobs in the WSN as follows,
\begin{align}
C_{\pi^{W}}&=\mathbb{E} (i) \\
		&=\frac{\lambda_{1}}{\lambda_{1}+\lambda_{2}} \cdot \mathbb{E}(i+j) \notag\\
		&=\frac{\lambda_{1}}{\lambda_{1}+\lambda_{2}} \cdot \frac{\lambda_{1}+\lambda_{2}}{\mu-(\lambda_{1}+\lambda_{2})}  \notag \\
		&=\frac{\lambda_{1}}{\mu-(\lambda_{1}+\lambda_{2})} \notag
\end{align}

We now analyze the expected assignment cost under the policy $\pi^{Db}$.

We define the cost of the policy $\pi^{Db}$ in terms of the limiting probabilities as follows,
\begin{align}\label{cost}
C_{\pi^{Db}}&=\lambda_{1} \sum\limits_{N \geq T} \pi_{N}(N) \cdot (N-T)^{+}, 
\end{align}
where $\pi_{N}(N)=\sum\limits_{j} \pi(j,N)$ is the long-run proportion of time that the process is in state $N$.

We have the following balance equations for component $j$, 

\begin{equation}\label{BalanceEqJ}
\begin{cases}
\pi_{j}(0)=\mu \pi_{j}(1)+ (1-\lambda_{2}) \pi_{j}(0)\\
\pi_{j}(1)=\mu\pi_{j}(2) +(1-\lambda_{2}-\mu)\pi_{j}(1) +\lambda_{2}\pi_{j}(0)\\
\pi_{j}(N-1)=\mu\pi_{j}(N) +(1-\lambda_{2}-\mu)\pi_{j}(N-1) +\lambda_{2}\pi_{j}(N-2)\\
\sum\limits_{k}\pi_{j}(k)=1
\end{cases}
\end{equation} where $\pi_{j}(0)=\sum\limits_{N} \pi(0,N)$.

Solving (\ref{BalanceEqJ}), we obtain: 

\begin{equation}\label{p0number}
\pi_{j}(0)=1-\frac{\lambda_{2}}{\mu}
\end{equation}

Notice that 
\begin{align}\label{eqPiRecurrent}
\pi(0,N)&=(1-\lambda_{2}-\mu)\pi(0,N-1)+\mu\pi(0,N-1)\\
		&=(1-\lambda_{2})^{N}\pi(0,0) \notag
\end{align}

Now
\begin{align}\label{pi0=pi00}
\pi_{j}(0)&=\sum\limits_{N} \pi(0,N) \notag \\
		&=\sum\limits_{N} (1-\lambda_{2})^{N}\pi(0,0) \notag \\
		&=\frac{1}{\lambda_{2}} \pi(0,0)
\end{align}

From (\ref{pi0=pi00}) and (\ref{p0number}), we obtain:
\begin{equation}\label{pi00}
\pi(0,0)=\frac{(\mu-\lambda_{2})\lambda_{2}}{\mu}
\end{equation}

We have the following balance equations for component $N$, 

\begin{equation}\label{BalanceEqN}
\begin{cases}
\pi_{N}(0)=(1-\mu)\pi_{N}(N-1) + \mu\pi_ {N}(0,N-1)\\
\pi_{N}(0)=\mu \sum \limits_{N}\pi(N) - \mu \sum \limits_{N} \pi(0,N), \textrm{ with }\pi(N)=\sum\limits_{j} \pi(j,N)\\
\sum\limits_{k}\pi_{N}(k)=1
\end{cases}
\end{equation}  where $\pi_{N}(N)=\sum\limits_{j} \pi(j,N)$.

But $ \sum \limits_{N}\pi(N)=1$ and $\sum \limits_{N} \pi(0,N)=\pi_{j}(0)=\frac{1}{\lambda_{2}} \pi(0,0)$ as per (\ref{pi0=pi00}).

Now (\ref{BalanceEqN}) becomes,

\begin{equation}\label{BalanceEqN2}
\begin{cases}
\pi_{N}(0)=(1-\mu)\pi_{N}(N-1) + \mu\pi_ {N}(0,N-1)\\
\pi_{N}(0)=\mu [1- \frac{1}{\lambda_{2}} \pi(0,0)]\\
\sum\limits_{k}\pi_{N}(k)=1
\end{cases}
\end{equation}

Solving for (\ref{BalanceEqN2}), we have that 

\begin{equation}\label{piN}
\pi_{N}(N)=\lambda_{2}(1-\lambda_{2})^{N}
\end{equation}

Using (\ref{piN}), we can now compute the cost (\ref{cost}) as follows,
\begin{align*}\label{cost}
C_{\pi^{Db}}&=\lambda_{1} \sum\limits_{N \geq T} \pi_{N}(N) \cdot (N-T)^{+} \\
		&=\lambda_{1} \sum\limits_{N \geq T} \lambda_{2}(1-\lambda_{2})^{N}  \cdot (N-T)^{+}\\
		&=\lambda_{1} \sum\limits_{N' \geq 0} \lambda_{2}(1-\lambda_{2})^{N'+T} \cdot N'\\
		&=\lambda_{1} \lambda_{2} (1-\lambda_{2})^{T} \sum\limits_{N' \geq 0}(1-\lambda_{2})^{N'} \cdot N'\\
		&=\lambda_{1} \lambda_{2} (1-\lambda_{2})^{T+1} \sum\limits_{N' \geq 0}(1-\lambda_{2})^{N'-1} \cdot N'\\
		&=\lambda_{1} \lambda_{2} (1-\lambda_{2})^{T+1} (-\frac{1}{\lambda_{2}})'\\
		&=\frac{\lambda_{1}}{\lambda_{2}} (1-\lambda_{2})^{T+1}
\end{align*}
\end{proof}

\pagebreak
\section{ Optimal Policy under different values of the uniformization parameter}
\label{B}

The structure of the optimal policy for various values of the uniformization parameter $B$ remains the same (see Figure \ref{VariousB}). The threshold tolerance is set to $T=1$.
%While for small uniformization parameters $B$ the optimal policy has a sightly different shape (see cases (a) and (b)), the policy converges for larger $B$. 
%This discrepancy is due to the fact that we discretized the state space of the Discrete Time Markov Decision Process into exponential phases with parameter $B$. 
%As the uniformization parameter $B$ increases, the structure of the optimal policy remains the same.
\begin{figure}[h]
  \begin{center}
    \subfigure[$B=\lambda_1+\lambda_2+\mu, N=30$]{\label{fig:edge-a}\includegraphics[scale=0.55]{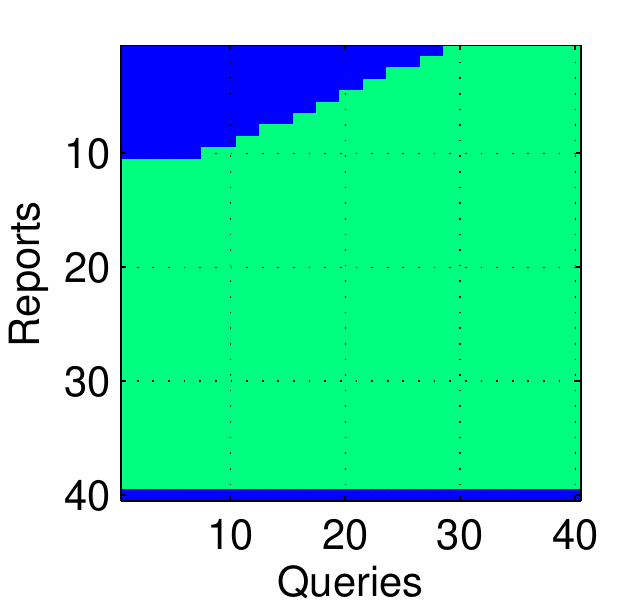}} \:\:\:\:\:\:\:\:\:\: \:\:\:\:\:\:\:\:\:\:
    \subfigure[$B=2(\lambda_1+\lambda_2+\mu), N=60$]{\label{fig:edge-b}\includegraphics[scale=0.55]{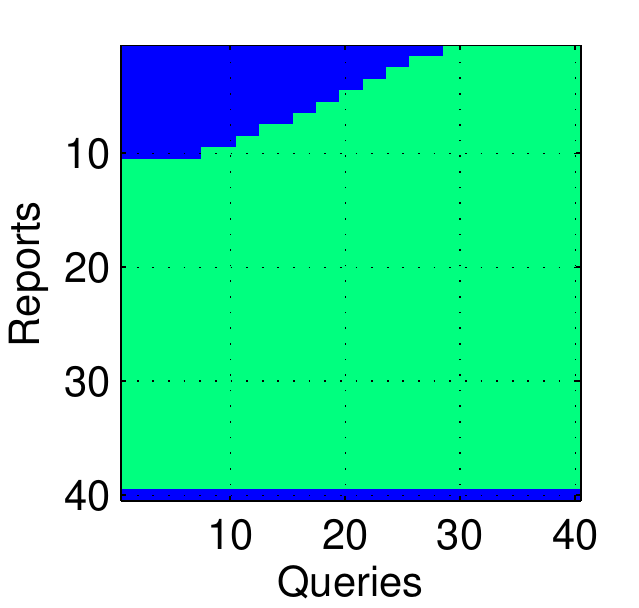}}  \\
    \subfigure[$B=5(\lambda_1+\lambda_2+\mu), N=150$]{\includegraphics[scale=0.55]{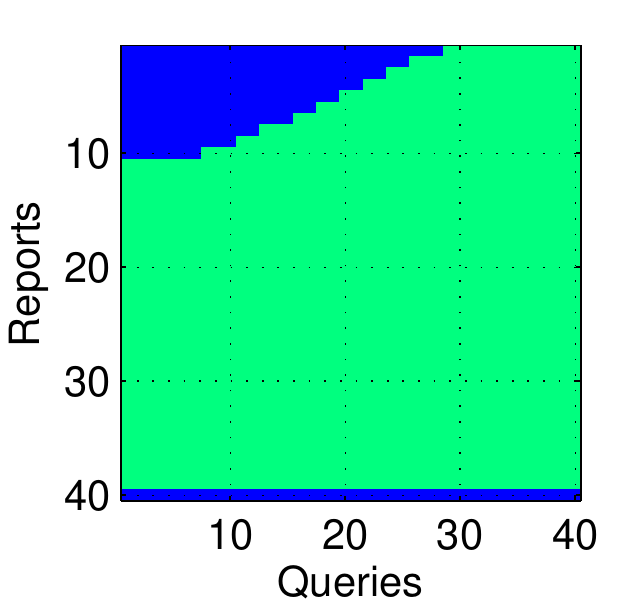}} \:\:\:\:\:\:\:\:\:\: \:\:\:\:\:\:\:\:\:\:
    \subfigure[$B=10(\lambda_1+\lambda_2+\mu), N=300$]{\label{fig:edge-c}\includegraphics[scale=0.55]{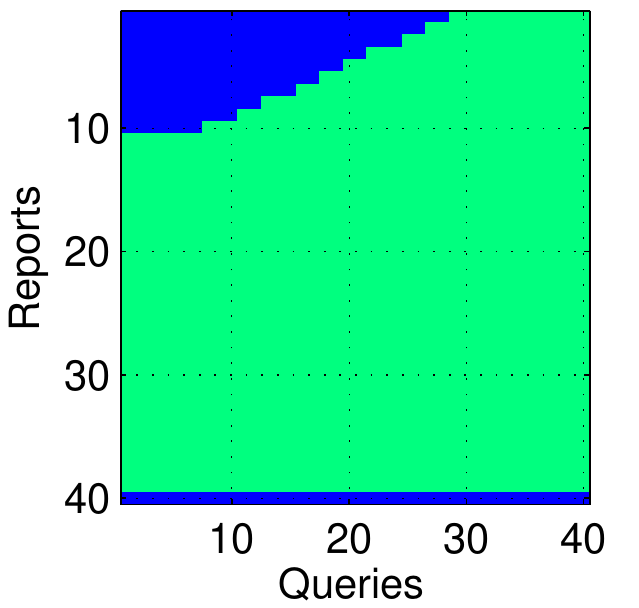}} 
  \end{center}
  \caption{Various uniformization parameter $B$. WSN assignment (blue) and DB assignment (green).}
  \label{VariousB}
\end{figure}

%% References
%%
%% Following citation commands can be used in the body text:
%% Usage of \cite is as follows:
%%   \cite{key}          ==>>  [#]
%%   \cite[chap. 2]{key} ==>>  [#, chap. 2]
%%   \citet{key}         ==>>  Author [#]

%% References with bibTeX database:

%\bibliographystyle{model1a-num-names}
%\bibliographystyle{unsrt}
\bibliographystyle{plain}
\bibliography{bibtrue}

%% Authors are advised to submit their bibtex database files. They are
%% requested to list a bibtex style file in the manuscript if they do
%% not want to use model1a-num-names.bst.

%% References without bibTeX database:

% \begin{thebibliography}{00}

%% \bibitem must have the following form:
%%   \bibitem{key}...
%%

% \bibitem{}

% \end{thebibliography}

\end{document}